\documentclass[12pt]{amsart}
\usepackage{amsfonts,amssymb,latexsym,amsmath, amsxtra}
\usepackage[all]{xy}
\usepackage[dvips]{graphics}
\usepackage{color}
\theoremstyle{plain}
\newtheorem{theorem}{Theorem}[section]

\newtheorem{conjecture}[theorem]{Conjecture}
\newtheorem{lemma}[theorem]{Lemma}

\newtheorem*{conjecture*}{Conjecture}

\theoremstyle{definition}

\theoremstyle{remark}
\newtheorem*{remark}{Remark}

\numberwithin{equation}{section}

\newcommand{\R}{\mathbb R}
\newcommand{\N}{\mathbb N}
\newcommand{\Z}{\mathbb Z}

\def\H{\mathbb H}

\newcommand{\Q}{{\mathbb Q}}

\def\vt{\vartheta}

\def\({\left(}
\def\){\right)}

\newcommand{\ol}[1]{\overline{{#1}}}

\newcommand{\re}[1]{\text{Re}\(#1\)}
\newcommand{\im}[1]{\text{Im}\(#1\)}

\newcommand{\abs}[1]{\left|#1\right|}
\newcommand{\wh}[1]{\widehat{#1}}
\newcommand{\wt}[1]{\widetilde{#1}}

\def\SO{\text{SO}}
\def\k2{\frac{k}{2}}

\begin{document}

\title[Partial and Mock Theta Functions]{
A Unified Partial and Mock Theta Function}

\author{Robert C. Rhoades}
\address{Department of Mathematics \\Stanford University \\ Stanford, CA 94305 \\ U.S.A. }
\email{rhoades@math.stanford.edu}


\begin{abstract} 
The theta functions $$\sum_{n\in \Z} \psi(n) n^\nu e^{2\pi i n^2 z},$$ with $\psi$ a Dirichlet character and $\nu = 0,1$,
have played a significant role in 
the theory of holomorphic modular forms and modular $L$-functions.   A partial theta functions are 
defined a sum over part of the integer lattice, such as $\sum_{n>0} \psi(n) n^\nu e^{2\pi i n^2 z}$. Such 
sums fail to have modular properties.  We give a construction which unifies these partial theta functions 
with the mock theta function introduced by Ramanujan.  The modularity of Ramanujan's mock theta functions 
has only recently been understood by the work of Sander Zwegers.  
\end{abstract}
\maketitle


\section{Introduction and Main Result}\label{sec:Result}
Shimura \cite{shimura} constructed the theta functions 
$$\theta(\psi, \nu;z) = \sum_{n\in \Z} \psi(n) n^\nu e(n^2 z)$$
where $\psi$ is a primitive Dirichlet character of conductor $r$ satisfying 
$\psi(-1) = (-1)^\nu$ and $\nu = 0$ or 1 with $e(z):= e^{2\pi i z}$ and $z$ in the upper half plane.
He proved (Proposition 2.2 of \cite{shimura}) that this theta series is a holomorphic 
modular forms of weight $1/2$ on $\Gamma_0(4r^2)$ with Nebentypus $\psi$ when $\nu =0$ and 
$\psi \cdot \chi_{-4}$ when $\nu = 1$ 
with $\chi_{-4}$ the nontrivial Dirichlet character modulo 4.
These theta series have played a significant role in the development of modular forms and in particular the 
development of half integral weight modular forms and modular $L$-functions.

A partial theta function is a function 
$$\theta^*(\psi, \nu; z) := \sum_{n\ge 0} \psi(n) n^\nu e\(n^2z\)$$ where $\psi$ and $\nu$ are as above,
but  $\psi(-1) = (-1)^{\nu+1}$.  While these functions are not modular, they have played a significant role in the theory
of quantum invariants of $3$-manifolds \cite{lz, hikami} and Vassiliev  knot invariants \cite{zagierStrange, zagierHalf}.

In this note we give a construction that unites these partial theta functions and the mock theta 
functions of Ramanujan.

Following Zagier \cite{zagierBourbaki}, a mock theta function is a $q$-series $H(q) = \sum_{n=0}^\infty a_n q^n$
such that there exists a rational number $\lambda$ and a unary theta function of weight $3/2$, 
$g(z) = \sum_{n>0} b_nq^n$, such that setting $q=e^{2\pi i z}$, then  
$h(z) = q^{\lambda} H(q) + g^*(z)$ non-holomorphic modular form of weight $1/2$, where 
$$g^*(z) = \int_{-\ol{z}}^{i\infty} \frac{ \ol{g(-\ol{\tau})} d\tau}{\sqrt{\tau+z}}.$$
The theta function $g$ is called the \emph{shadow} of the mock theta function $H$.

The mock theta functions  introduced by Ramanujan are given by $q$-hypergeometric 
series.  
For instance, 
the third order mock theta function $f$ may be defined for $\abs{q}<1$ by  
\begin{equation}
f(q) := \sum_{n=0}^\infty \frac{q^{n^2}}{(-q;q)_n^2} 
\end{equation}
It has gain much attention in 
recent years because of its connection to the rank statistic for partitions \cite{bo1, bo}.  It is also intimately related
to the generating function for the number of strictly unimodal sequences \cite{rhoades}. 
The shadow of the mock theta function $f(q)$ is proportional to 
$g_f(z) := \sum_{n\in \Z} \(\frac{-12}{n}\) n q^{\frac{n^2}{24}}$. 
Thus its pseudomodularity is related to the pseudomodularlity
properties of the Eichler integral 
$\int_{-\ol{z}}^{i\infty} \frac{ g_f(\tau) d\tau}{\sqrt{\tau +z}}$

Ramanujan's third order mock theta function may also be defined  for $\abs{q}<1$
by the $q$-hypergeometric series
$$f_2(q) := 1+ \sum_{n\ge 1} \frac{(-1)^n q^n}{(1+q)(1+q^2)\cdots(1+q^n)}.$$
That is for $\abs{q}<1$
$$f_2(q) = f(q)$$
see, for instance, \cite{hikami} (4.20) or the introduction of \cite{BFR}. 


A striking property of the $q$-hypergeometric series defining $f(q)$ and $f_2(q)$ is that they are equal for 
$\abs{q}<1$, 
but are not equal for $\abs{q}>1$. 
Namely, 
for $\abs{q}>1$ we have (see \cite{hikami})
$$f_2 \(q\):= 1+ \sum_{n\ge1} \frac{(-1)^n q^{-\frac{n(n+1)}{2}}}{(1+q^{-1})\cdots (1+q^{-n})} = 2\sum_{n\ge 1} 
\(\frac{-12}{n}\) q^{-\frac{n^2-1}{24}} =:2\psi(q^{-1}).$$
On the other hand, for $\abs{q}>1$ (see \cite{andrewsLost})  
\begin{align*}
f(q) =& 1+\sum_{n\ge1} \frac{q^{-n}}{(1+q^{-1})^2\cdots(1+q^{-n})^2} 
= 2\psi(q^{-1}) - \frac{1}{(-q^{-1};q^{-1})_\infty^2}  \sum_{n\ge 0} (-1)^n q^{-\frac{n(n+1)}{2}}
\end{align*}
\begin{remark}
See \cite{BFR} for details and many more examples of 
$q$-hypergeometric series which are equal to mock theta functions
in the domain $\abs{q}<1$ and not equal but related to partial theta functions in the domain $\abs{q}>1$. . 
\end{remark}

The non-uniqueness of $q$-hyperegeometric series representations for mock theta functions 
is problematic in continuing the function from the inside of the disc to the 
outside of the disc.   Moreover, there may not be a $q$-hypergeometric representation for every mock 
theta function.  

In this paper we give an analytic construction of a single function that equals the mock theta function $f(q)$ in the 
upper half plane and the partial theta function $2\psi(q)$ in the lower half plane.  
\begin{remark} 
According to our definition, $q \psi(q^{24})$ is a partial
theta function. 
\end{remark}

Define $$a_k(s) := \sum_{m\ge 0} \(\frac{\pi}{12 k}\)^{2m+\frac{1}{2}} \frac{1}{\Gamma\(m+\frac{3}{2}\)}  
\frac{1}{s^{m+1}}$$
and 
$$\Phi_{d,k}(z):= \frac{1}{2\pi i}  \int_{\abs{s} = r} \frac{a_k(s) e^{23s}}{1- \zeta_{2k}^d q e^{24s}} ds$$
where $r$ is taken sufficiently small so that $\abs{ \log \( \zeta_{2k}^dq \)} \gg r$ and the integral converges.  
Furthermore, define 
\begin{align}
\omega_{h,k} := \exp\(\pi i s(h,k)\). \end{align}
Here we follow the standard notation for Dedekind sums, namely
$$
s(h,k) := \sum_{\mu \pmod{k}} \(\(\frac{\mu}{k}\)\)\(\(\frac{h\mu}{k}\)\),$$
with the sawtooth function defined as
$$\(\(x\)\) := \begin{cases} x - \lfloor x \rfloor -\frac{1}{2} &
\text{ if } x\in \R \setminus \Z,\\
0 & \text{ if } x\in \Z.\end{cases}
$$

\begin{theorem}\label{thm:main} Let $q = e^{2\pi i z}$. 
The function
$$F(z):= 1+ \pi \sum_{k=1}^\infty \frac{(-1)^{\lfloor \frac{k+1}{2}\rfloor}}{k} \sum_{d\pmod{2k}} \omega_{-d, 2k}
e\(-\frac{d}{8}(1+(-1)^k) + \frac{d}{2k} + z\)  \Phi_{d,k}(z)$$
converges for $z\in \H$ and $z\in \H^{-} = \{z: \im{z} < 0\}$.  Moreover, 
when $z\in \H$ we have $F(z) = f(q)$ and when $z\in \H^{-}$ we have 
$F(z) = 2\psi\(q^{-1}\)$.
\end{theorem}
\begin{remark}
Similar theorems exist for any partial theta function for a nontrivial character $\psi$.  
This theorem explains that partial theta functions may be constructed as lower half plane analogous of the 
mock theta functions.  The mock theta functions are completed to 
a non-holomorphic modular form by the addition of an Eichler integral. 
It would be interesting to find a similar completion for the partial theta functions.
\end{remark}

Our theorem relies on the construction of a Maass-Poincar\'e series for the mock theta function and 
the ``expansion of zero'' principle of Rademacher \cite{rademacher1} 
(see, for instance, Chapter IX of Lehner's book on Discontinuous Groups \cite{lehnerBook}).  
Rademacher proved an exact formula for $p(n)$, the number of partitions of $n$.  Using his 
formula he found an extension of the generating function to the lower half plane.  
Rademacher conjectured and later proved \cite{rademacherBook} that each of the Fourier coefficients 
of the function in the lower half plane is zero.  Rademacher's conjecture was proved independently 
by Petersson \cite{petersson}. 
Such expansions were noticed earlier by Poincar\'e.   
See his memoir on Fuchsian groups \cite{poincare} or 
the english Translation of Poincar\'e's paper by Stillwell \cite{poincareBook} 
(p. 204).   Extensions of the ``expansion of zero'' principle was 
written about by Lehner \cite{lehner}. Additionally, 
Knopp \cite{knopp} wrote about this principle in connection with Eichler cohomology.  
The perspective of Knopp's work is relevant here when one makes the connection between 
mock theta functions  $H(q)$
and their completions $h(q) = q^{\lambda} H(q) + g^*(z)$.   We do not address this connection here,
 but we hope to take it up in future work.

\section{Preliminaries}\label{sec:Background}

In Zwegers notation \cite{zwegers1, zagierBourbaki} let
$$h_3(z) = q^{-\frac{1}{24}} f(q)$$
and set $$R_3(z) := 
\frac{i}{\sqrt{3}} \int_{-\ol{z}}^\infty \frac{g_3(\tau)}{\sqrt{(\tau+z)/i}} d\tau$$
where
$$g_3(z) = \sum_{n\equiv 1\pmod{6}} nq^{\frac{n^2}{24}} = \sum_{n=1}^\infty \(\frac{-12}{n}\) n q^{\frac{n^2}{24}}.$$

Applying the straightforward calculation 
$$\int_{-\ol{z}}^{i\infty} \frac{e^{2\pi i \tau \frac{n^2}{24}}}{\sqrt{-i(\tau +z) } } d\tau  = i \(\frac{12}{\pi }\)^{\frac{1}{2}} n^{-1} 
\Gamma\(\frac{1}{2}, \frac{\pi n^2 y}{6}\) q^{-\frac{n^2}{24}} $$ we may rewrite $R_3$ as  
\begin{equation}\label{eqn:R3}
R_3(z) =  - 2\sum_{n=1}^\infty \(\frac{-12}{n}\)\pi^{-\frac{1}{2}} \Gamma\(\frac{1}{2}, \frac{\pi n^2 y}{6}\)
q^{\frac{-n^2}{24}}.
\end{equation}
Then the corrected function 
$$\wh{h_3}(z) = h_3(z) + R_3(z)$$ is a weight $\frac{1}{2}$ harmonic Maass form with respect to $\Gamma(2)$
(see  \cite{zagierBourbaki} page 07).

By work of Bringmann and Ono \cite{bo1} we may write $f$ as a Poincare series. 
Define the Kloosterman-like sum by  
$$A_k(n) = \sum_{x\pmod{k}} \omega_{-x,k}\cdot e\(\frac{nx}{k}\)$$
where the sum is over those $x$ relatively prime to $k$.
\begin{theorem}[Theorem 3.2 and Section 5 of \cite{bo1}]\label{thm:poincare} 
In the notation above we have $\wh{h_3}(z) = P_h(z) + P_{nh}(z)$ where 
\begin{align*}
P_h(z) := \  q^{-\frac{1}{24}} + \sum_{n=1}^\infty \alpha(n) q^{n-\frac{1}{24}}  \hspace{.1in}  \text{ and } \hspace{.1in} 
P_{nh}(z):=  - \pi^{-\frac{1}{2}}\Gamma\(\frac{1}{2}, \frac{\pi y}{6}\) q^{-\frac{1}{24}} + \sum_{n=-\infty}^0 \gamma_y(n) q^{n-\frac{1}{24}} 
 \end{align*}
 where 
 $$\alpha(n) = \frac{\pi}{ (24n-1)^{\frac{1}{4}}} \sum_{k=1}^\infty \frac{(-1)^{\lfloor \frac{k+1}{2}\rfloor} 
A_{2k}\(n-\frac{k(1+(-1)^{k})}{4}\)}{k} ) I_\frac{1}{2}\(\frac{\pi \sqrt{24n-1}}{12 k}\)$$
and
 \begin{align*}
 \gamma_y(-n) =& \pi^{-\frac{1}{2}} \Gamma\(\frac{1}{2}, \frac{\pi \abs{24n+1} y}{6}\) \\ & \times  
\frac{\pi}{(24n+1)^{\frac{1}{4}}} \sum_{k=1}^\infty \frac{(-1)^{\lfloor \frac{k+1}{2}\rfloor}}{k} 
A_{2k}\(-n-\frac{k(1+(-1)^k)}{4}\) J_\frac{1}{2}\(\frac{\pi \sqrt{24n+1}}{12k}\).
\end{align*}
\end{theorem}
\begin{remark}
The function $P_h$ is a holomorphic function while the function $P_{nh}$ is a non-holomorphic function.  This 
explains the subscripts. 
\end{remark}

From this we may deduce the following lemma. 
\begin{lemma}\label{lem:gammay}
Define 
$$\wt{\alpha}(n):= \frac{\pi}{(24n+1)^{\frac{1}{4}}} 
\sum_{k\ge 1}\frac{(-1)^{\lfloor \frac{k+1}{2}\rfloor}}{k} A_{2k}\(-n-\frac{k(1+(-1)^k)}{4}\) 
J_\frac{1}{2}\(\frac{\pi \sqrt{24n+1}}{12 k}\).$$
For $\wt{\alpha}(0) = 1$ and for $\ell = \frac{n^2-1}{24}> 0$  
we have 
$$-2\(\frac{-12}{n}\) =\wt{\alpha}(\ell).$$
Furthermore, $\wt{\alpha}(\ell) = 0$ in all remaining cases. 
\end{lemma}
\begin{proof}
Comparing from the definition of $\wh{h_3}$ and the series
for the coefficients $\gamma_y(n)$ in Theorem \ref{thm:poincare} we have
$$q^{\frac{1}{24}} R_3(z) = -\pi^{-\frac{1}{2}}\Gamma(\frac{1}{2}, \frac{\pi y}{6}) + \sum_{n =0}^\infty 
\gamma_y(-n) q^{-n}.$$
Using \eqref{eqn:R3} and the series expansion of $\gamma_y(n)$ 
we see that  $\wt{\alpha}(0) = 1$ and for $\ell >0$ we have  
$\gamma_y(-\ell) =0$ unless $\ell = \frac{n^2-1}{24}$ for some $n\equiv 1, 5 \pmod{6}$, 
that is $24 \ell+1 = n^2$ for some $n\ge 1$. For such $\ell$ 
we have 
$-2\(\frac{-12}{n}\) =\wt{\alpha}(\ell).$
\end{proof}

\section{Proof of Theorem \ref{thm:main}}\label{sec:Proof}

Throughout this section, for $c \in \N$ let $\zeta_c := e^{2\pi i \frac{1}{c}}$ be a root of unity. 
We begin by showing that $f(q)$  equals $F(q)$ for $\abs{q}<1$. 
We apply Theorem \ref{thm:poincare} and switch the order of summation to obtain 
\begin{align*}
f(q) =&   1+ \pi\sum_{n=1}^\infty  (24n-1)^{-\frac{1}{4}} \sum_{k=1}^\infty \frac{(-1)^{\lfloor \frac{k+1}{2}\rfloor} 
A_{2k}(n-\frac{k(1+(-1)^{k})}{4})}{k}) I_\frac{1}{2}\(\frac{\pi \sqrt{24n-1}}{12 k}\) e\(n(x+iy)\) \\
=& 1+ 
\pi \sum_{k=1}^{\infty} \frac{(-1)^{\lfloor \frac{k+1}{2}\rfloor} }{k} 
\sum_{n\ge 1} \frac{A_{2k} \(n-\frac{k(1+(-1)^{k}) }{4} \) } {(24n-1)^{\frac{1}{4}}} I_\frac{1}{2}\(\frac{\pi \sqrt{24n-1}}{12 k}\) 
e\(n (x+iy)\).
\end{align*}
\begin{remark} The order we have switched to is in some sense more natural.  To obtain the Fourier expansion, one 
completes the switch in the other direction. For instance see  the proof of Theorem 3.2 of \cite{bo1}. 
\end{remark}

Next we open the sum in the definition of $A_k$ to obtain 
\begin{align*}
f(q) = 1+\pi \sum_{k=1}^\infty \frac{(-1)^{\lfloor \frac{k+1}{2}\rfloor} }{k}   
\sum_{d\pmod{2k}}&\omega_{-d,2k} e\(-\frac{d (1+(-1)^k)}{8} \)   \\
&\times \sum_{n\ge 1} \frac{e\(n \( \frac{d}{2k} + x+iy\) \) }{ (24n-1)^{\frac{1}{4} }}
I_\frac{1}{2}\(\frac{\pi \sqrt{24n-1}}{12k}\).
\end{align*}

We begin by rewriting the function
$$S(\zeta_{2k}^d q):= \sum_{n\ge 1} \frac{e\(n \( \frac{d}{2k} + x+iy\) \) }{ (24n-1)^{\frac{1}{4} }}
I_\frac{1}{2}\(\frac{\pi \sqrt{24n-1}}{12k}\).$$ Later we will  
continue this function 
for values with $y<0$.
Write  $\tilde{q}:= \zeta_{2k}^d q$
so that 
$$S(\tilde{q}) = \sum_{n\ge 1} \frac{1}{ (24n-1)^{\frac{1}{4} }}
I_\frac{1}{2}\(\frac{\pi \sqrt{24n-1}}{12k}\)\tilde{q}^n.$$
and set
$$B_k(24t-1):=   \frac{1}{ (24t-1)^{\frac{1}{4} }}
I_\frac{1}{2}\(\frac{\pi \sqrt{24t-1}}{12k}\).$$

\begin{lemma} In the notation above, we have 
$$B_k(t) = \sum_{m\ge 0} \frac{b_m(k)}{m!}t^m$$
with $b_m(k) = O\(\frac{C^m}{m!}\)$.  Moreover, 
$b_m(k) = \(\frac{\pi}{12k}\)^{2m+\frac{1}{2}} \frac{1}{\Gamma\(m+\frac{3}{2}\)}.$ 
\end{lemma}
\begin{proof}
The proof follows from the fact that $I_\frac{1}{2}(x) = i^{-\frac{1}{2}} J_\frac{1}{2}(ix)$ and 
the Taylor expansion of $J_\frac{1}{2}(ix)$ is given by 
$$\sum_{m=0}^\infty \frac{(-1)^m}{m! \Gamma\(m+\frac{3}{2}\)} \(\frac{i}{2} x\)^{2m+\frac{1}{2}}.$$
Hence 
$$I_\frac{1}{2}\(\frac{\pi}{12k}\sqrt{t}\) = (t)^\frac{1}{4} \sum_{m\ge 0} \(\frac{\pi}{12 k}\)^{2m+\frac{1}{2}} \frac{1}{
m! \Gamma\(m+\frac{3}{2}\)} t^m.$$
\end{proof}

A standard calculation (Bernoulli-Laplace transform?)
gives 
$$B_k(t) = \frac{1}{2\pi i} \int_{\abs{s} = r} e^{st}\sum_{m\ge 0} \frac{b_{m}(k)}{s^{m+1}} ds$$
where $r$ may be taken to sufficiently small.
We have $$a_k(s) = \sum_{m\ge 0} \frac{b_{m}(k)}{s^{m+1}}.$$
By the previous lemma this series is absolutely convergent for all $s$. 
We have
\begin{align}
S(\tilde{q}) = \sum_{n\ge 1} B_{k}(24n-1) (q\zeta_{2k}^d)^n =& \frac{1}{2\pi i} \int_{\abs{s}=r}
 \sum_{n\ge 1} e^{s(24n-1)}  
(q\zeta_{2k}^d )^n  a_k(s) ds  \\
=& \frac{1}{2\pi i} \zeta_{2k}^d q \int_{\abs{s} =r} \frac{a_k(s)e^{23s}  ds}{1-\zeta_{2k}^d q e^{24s}}. \label{eqn:S}
\end{align}
where now we point out that $r$ can be taken to be small enough so that $\abs{e\(x+iy\) e^{24s}} <1$.

Thus we have established the first part of the Theorem, namely
$$f(q) = 1+ \pi \sum_{k=1}^\infty \frac{(-1)^{\lfloor \frac{k+1}{2}\rfloor}}{k} \sum_{d\pmod{2k}} \omega_{-d, 2k}
e\(-\frac{d}{8}(1+(-1)^k) + \frac{d}{2k} + z\)  \Phi_{d,k}(z)$$
where we recall that 
$$\Phi_{d,k}(z) = \frac{1}{2\pi i}  \int_{\abs{s} = r} \frac{a_k(s) e^{23s}}{1- \zeta_{2k}^d q e^{24s}} ds.$$

To prove the second part of the theorem we return to the the integral in \eqref{eqn:S}.
This integral converges so long as $- \frac{1}{24}\log(\zeta_{2k}^d q)$ 
is outside the circle we integrate over, 
that is so long as $\abs{\log(\zeta_{2k}^d q)} > 24r$. When $\zeta_{2k}^d q \ne 1$ 
we can find $r$ small enough so that this 
integral converges.  In particular $\Phi_{d,k}(z)$ is 
regular in the entire complex sphere except at $\zeta_{2k}^d q = 1$.

Now we will construct the Fourier expansion in the region where $\abs{q}>1$.   As above, 
with $\tilde{q}:= \zeta_{2k}^{d} q$, 
we take 
$r$ sufficiently small so that $\abs{\tilde{q} e^{24s}}>1$ and obtain 
\begin{align*}
\tilde{q} \Phi_{d,k}(z) = \sum_{n\ge 0} B_{k}(24n-1) (q\zeta_{2k}^d)^n 
=& \sum_{m\ge 0}   b_m(k)\frac{\tilde{q}}{2\pi i} \int_{\abs{s} =r} \frac{e^{23s} ds}{s^{m+1}(1-\tilde{ q} e^{24s})} \\
=&- \sum_{m\ge 0} b_m(k) \int_{\abs{s}=r} \frac{ds}{s^{m+1}} \sum_{n=0}^\infty (\tilde{q})^{-n} e^{-(24n+1)s}  \\
=& -\sum_{m\ge 0} b_m(k)  \sum_{n=0}^\infty (\tilde{q})^{-n}   \int_{\abs{s}=r} \frac{ds}{s^{m+1}} e^{-(24n+1)s}  \\
=&- \sum_{m\ge 0} \frac{b_m(k)}{m!}\sum_{n=0}^\infty ( \tilde{q})^{-n}  (-24n-1)^{m}\\
=& -\sum_{n=0}^\infty \zeta_{2k}^{-dn} q^{-n} \sum_{m\ge 0} \frac{b_m(k)}{m!} (-24n-1)^m\\
=& -\sum_{n=0}^\infty   \zeta_{2k}^{-dn} 
q^{-n} \frac{1}{(-24n-1)^{\frac{1}{4}}} I_\frac{1}{2}\(\frac{\pi \sqrt{-24n-1}}{12k}\) \\
=& -\sum_{n=0}^\infty  \zeta_{2k}^{-dn} q^{-n} \frac{(-1)^{\frac{1}{2}} i^{-\frac{1}{2} } }{(-(24n+1))^{\frac{1}{4}}} 
J_{\frac{1}{2}}\(\frac{\pi \sqrt{24n+1}}{12k}\) 
\end{align*}
where we have used 
$$I_\frac{1}{2}(ix) = i^{-\frac{1}{2}} J_\frac{1}{2}(-x) = i^{-\frac{1}{2}} (-1)^\frac{1}{2} J_{\frac{1}{2}}(x).$$
But then the Fourier expansion of $F(q)$ in the \emph{lower half-plane} becomes 
\begin{align*}
F(q) =&  1- \pi \sum_{k=1}^\infty \frac{(-1)^{\lfloor \frac{k+1}{2}\rfloor} }{k}   
\sum_{d\pmod{2k}}\omega_{-d,2k} e\(-\frac{d (1+(-1)^k)}{8} \) \\
&\hspace{2in} \times \sum_{n=0}^\infty \zeta_{2k}^{-dn} q^{-n} 
(24n+1)^{-\frac{1}{4}}  J_{\frac{1}{2}} \(\frac{\pi}{12 k} \sqrt{24n+1}\)  \\ 
=& 1 - \pi \sum_{k=1}^\infty \frac{ (-1)^{\lfloor \frac{k+1}{2}\rfloor}} {k} \sum_{n=0}^\infty 
\frac{A_{2k}\(  -n - \frac{k(1+(-1)^k)}{4}\)}{(24n+1)^{\frac{1}{4}}}  J_{\frac{1}{2}} \(\frac{\pi}{12 k} \sqrt{24n+1}\) q^{-n}\\
=& 1-\sum_{n\ge 0} q^{-n} \wt{\alpha}(n)
\end{align*}
where the last equality follows from switching the order of summation.  Also, we note that convergence of the 
series $\tilde{\alpha}$ follows from Theorem 4.1 of \cite{bo1}.

Applying Lemma \ref{lem:gammay} we may deduce that 
for $\abs{q}>1$ we have $F(q) = 2\sum_{n\ge 1} \(\frac{-12}{n}\)  q^{-\frac{n^2-1}{24}}.$

\section{Relationship to Other Works}
The connection between partial theta functions and mock theta functions has been observed in the work of 
Lawerence-Zagier \cite{lz} and Zwegers \cite{zwegers}.  In this section we describe the framework 
layed out in their works.  It remains an open question as to if and how the construction of this work might 
shed light on the applications those authors had in mind.

\subsection{WRT-Invariants}\label{sec:invariants}
Curiously partial theta functions have arisen in the computation of topological invariants.  See, for instance, \cite{lz, zagierStrange, zagierHalf}. 
For instance, for each root of unity $\xi$ the Witten-Reshetikhin-Turaev (WRT) invariant associated to the Poincar\'e homology sphere 
$M$ is an element $W(\xi) \in \Z[\xi]$.   The Poincar\'e homology sphere is the quotient 
space $X  = \SO(3)/\Gamma$ where 
$\Gamma$ is the rotational symmetry group of the icosahedron.  
Thus, $X$ is a 3-manifold $X$ which has the same homology as a $3$-sphere, namely
$H_0(X, \Z) = H_3(X, \Z) = \Z$ and $H_n(X, \Z) = \{0\}$ for all other $n$. 

Let  $$A_{\pm} (q) := 
\sum_{\begin{subarray}{c} n>0 \\ n\equiv \pm 1\pmod{5}\end{subarray}}  \(\frac{12}{n}\) q^{(n^2-1)/120},$$
where $\(\frac{\cdot}{\cdot}\)$ is the Legendre symbol. 
Lawerence and Zagier (Theorem 1 and the remark preceding (20)  of \cite{lz}) 
proved that the radial limit as $q\to \xi$  (for $\abs{q}<1$) of 
$1-A_{\pm}(q)$
agrees with $W(\xi)$. 
The function 
$q^{\frac{1}{120}}(A_{-}(q) + A_{+}(q)) 
= \sum_{n>0} \re{\(\frac{12}{n}\) \epsilon(n)} q^{\frac{n^2}{120}},$ where $\epsilon$ is the nontrivial Dirichlet 
character of modulus 5, thus it is very nearly a partial theta function. 

Lawerence and Zagier explain that while the partial theta function 
$\sum_{n>0} \(\frac{12}{n}\) \epsilon(n) q^{\frac{n^2}{120}}$
does not have any modular properties 
its asymptotics toward roots of unity
are equal, up to a constant, to the asymptotics toward roots of unity of the Eichler integral
$$\Theta^*(z):= \int_{-\ol{z}}^{i\infty} \frac{\Theta(\tau) d\tau}{\sqrt{\tau + z}}$$
where $\Theta(z) = \sum_{n\in \Z} n \(\frac{12}{n}\) \epsilon(n) q^{\frac{n^2}{120}}$ is a weight $3/2$ theta function.   
As explained by Zwegers \cite{zwegers} 
this Eichler integral has pseudomodular transformations
which match the pseudomodular transformations of Ramanujan's mock theta functions. 

Moreover, Zwegers (see Section 5 of \cite{lz}) demostrated a curious connection between  one of
Ramanujan's mock theta functions and the functions $A_{\pm}(q)$.  Define 
\begin{align*}
\Phi(q):= -1+ \sum_{n=0}^\infty \frac{q^{5n^2}}{(q;q^5)_n (q^4;q^5)_n}.
\end{align*}
The $\Phi$, is a mock theta function. 
As in the introduction, the series defining $\Phi(q)$ converges not only for $\abs{q}<1$, but also for $\abs{q}>1$.  
Let $\Phi^*(q):= -1 - \sum_{n=0}^\infty \frac{q^{5n+1}}{(q;q^5)_n (q^4;q^5)_n} = \Phi(1/q)$.
Then Zwegers establishes 
$$- \Phi^*(q) = A_+(q) - \frac{1}{(q^4;q^5)_\infty (q;q^5)_\infty} F_+(q) = A_-(q) +  \frac{1}{(q^4;q^5)_\infty (q;q^5)_\infty}  
F_-(q)$$
where $F_{\pm}(q) := \sum_{\pm \(n-\frac{1}{2}\) > 0} (-1)^n q^{5n^2-n)/2}.$

These $q$-series identities establish a deep relationship between the WRT invariants, 
the asymptotics of partial theta functions, and the asymptotics of a mock theta function.   The Poincar\'e homology
sphere is an example of a Seifert manifold (a certain type of construction).   Similar results exist for 
all other Siefert manifolds, see \cite{hikami}. 


\subsection{$C^\infty$ asymptotics}\label{sec:Cinfity}
Zagier's definition of a mock theta function was not how Ramanujan thought of mock theta functions.  
In his final letter to Hardy, Ramanujan explained the concept of a mock theta function. Following 
Zwegers's slight rephrasing \cite{zwegers1}, Ramanujan ``defined''' 
a mock theta function is a function $f$ of the complex variable $q$, 
defined by a $q$-hypergeometric series (Ramanujan calls this the Eulerian form), which converges for $\abs{q}<1$
and satisfies the following 
\begin{enumerate}
\item infinitely many roots of unity are exponential singularities
\item for every root of unity $\xi$ there is a theta-function $\vt_\xi(q)$ such that the difference 
$f(q)-\vt_\xi(q)$ is bounded as $q\to \xi$ radially,
\item there is no theta function that works for all $\xi$, i.e. $f$ is not the sum of two functions one of which is a theta function 
and the other a function which is bounded in all roots of unity. 
\end{enumerate}
\begin{remark}
When Ramanujan refers to theta functions he means sums, products, and quotients of series of the form $\sum_{n\in \Z} 
\epsilon^n q^{an^2+bn}$ with $a, b\in \Q$ and $\epsilon = \pm 1$. 
\end{remark}


Zwegers made the following conjecture. 
\begin{conjecture}
If $\xi$ is a root of unity where $f$ is bounded (as $q\to \xi$ radially inside the unit circle), for example $\xi=1$, then 
$f$ is $C^\infty$ over the line radially through $\xi$. 

If $\xi$ is a root of unity where $f$ is not bounded, for example, $\xi=-1$, then the asymptotic expansion of the bounded 
term in condition (2) in the `definition' of mock theta function is the same as the asymptotic expansion of $f$ as 
$q\to \xi$ radially outside the unit circle. 
\end{conjecture}
\begin{remark}
Zwegers states this theorem for a different mock theta function of Ramanujan.  However, by his reasoning the same should 
follow for $f$. 
\end{remark}
The work of Lawerence and Zagier \cite{lz} shows that the limits must agree. 
It would be interesting to use the function $F$ from Theorem \ref{thm:main} of this paper to prove this conjecture. 



\end{document}